\documentclass{article}
\usepackage{amsmath,amsthm,amssymb,mathrsfs}

\begin{document}

\begin{titlepage}

\title{A note on the mean square of Riemann zeta-function}
\author{Li An-Ping\\Beijing 100086, P.R.China\\apli0001@sina.com}

\date{}

\maketitle 

\vspace{1cm}

\begin{abstract}
In this paper, we will give a new proof for a known result of the mean square of the Riemann zeta-function
\end{abstract}

% \tableofcontents

\end{titlepage}

\section{Introduction}

Let
\[
I=\int_{0}^{T}\lvert\zeta(1/2+it)\rvert^2\lvert A(1/2+it)\rvert^2dt
\]
where $\zeta(s)$ is the Riemann zeta-function, and 
\[A(s)=\sum_{1\le m\le M}a(m)m^{-s}\] 
is a finite Dirichlet series. For this type of mean square of $\zeta(s)$, there are a number of researches, 
which is related to estimate the number of the zeros of $\zeta(s)$ on the critical line $\operatorname{Re}(s)=1/2$, see [$1\sim6$].
\newline
\newline
As before, for two positive integers $h,k$, denote $h^{*}=h/(h,k),k^{*}=k/(h,k)$, and $\overline{h^{*}}$ is the least positive integer
such that $\overline{h^{*}}h^{*}\equiv 1\mod k^{*}$.
\newline
\newline
For $s=c+it,c=1+\eta,0<\eta<1$, denote by 
$$
\mathscr{M}(s)=\sum_{n=1}^{\infty}\frac{d(n)}{n^{s}}\sum_{h,k\le M}^{}\frac{a(h)\overline{a(k)}}{(hk)^{1-s}}(h,k)^{(1-2s)}e\left(\frac{n\overline {h^{*}}}{k^{*}} \right)
$$
and $V=\mathop{\sup}\limits_{|t|\le M}\{\mathscr{M}(c+it)\}$.
\newline
Balasubramanian, Conrey, and Heath-Brown proved that
\newline
For $a(m)\ll m^{\epsilon}$, and $\log(m)\ll \log T$, there is
\begin{equation}
I=T\sum_{h,k\le M}\frac{a(h)}{h}\frac{\overline{a(k)}}{k}(h,k)\left(\log\frac{T(h,k)^2}{2\pi hk}+2\gamma -1\right)+\mathscr{E}.\tag{1.1}
\end{equation}
where $\mathscr{E}\ll VT^{\epsilon}+o(T)$.\\
\newline
In the paper[1], it is applied an auxiliary function $w(t,T_1,T_2)$, which plays an important role of "ferrying", Instead, in this paper, we will make use of 
a new one, i.e. $\omega(t,T_1,T_2)$ defined in Lemma 2.1. Our result is that
\newline
\newtheorem{mythm}{Theorem}[section]

\begin{mythm}
 Suppose that $a(m)\ll m^{\epsilon}$, and $\log(m)\ll \log T$, then 
\begin{equation}
I=T\sum_{h,k\le M}\frac{a(h)}{h}\frac{\overline{a(k)}}{k}(h,k)\left(\log\frac{T(h,k)^2}{2\pi hk}+2\gamma -1\right)+\tilde{\mathscr{E}}.\tag{1.2} 
\end{equation}
where $\tilde{\mathscr{E}}\ll VT^{-\eta + \epsilon}+o(T)$.\\
\end{mythm}

As it is easy to know that $V \ll M^{2+2\epsilon+2\eta} $, then with a properly re-define $\epsilon,\eta$, it has
\begin{equation}
\tilde{\mathscr{E}},\mathscr{E}\ll M^2T^{\epsilon}.\tag{1.3}
\end{equation}
The main arguments in this paper are most similar to the ones in [1].

\section{The Proof of Theorem 1.1}

\vspace{20pt}

\newtheorem{mylem}{Lemma}[section]

\begin{mylem}
Suppose that $\lambda=T^{2-\epsilon}$, define
\begin{equation}
\omega(t,T_1,T_2)=\frac{e^{\lambda}}{2\pi}\int_{T_1}^{T_2}\Gamma(\lambda+(u-t)i)\lambda^{-(\lambda+(u-t)i)}du.\tag{2.1}
\end{equation}
Let $\alpha\ge10$, denote $\Delta=(2\alpha\lambda\log T)^{1/2}$, if $t\in [T_1+\Delta,T_2-\Delta]$, then
\begin{equation}
\left|\omega(t,T_1,T_2)-1\right|\ll T^{-\alpha}.\tag{2.2} 
\end{equation}
And if $t\le T_1-\Delta$, or $t\ge T_2+\Delta$, then
\begin{equation}
\left|\omega(t,T_1,T_2)\right|\ll T^{-\alpha}.\tag{2.3} 
\end{equation}

\end{mylem}

\begin{proof}
It is familiar that 
\[
e^{-\lambda}=\frac{1}{2\pi i}\int\limits_{(c)}\Gamma(s)\lambda^{-s}ds.
\] 
Hence,
\[
1-\omega(t,T_1,T_2)=R_1+R_2,
\]
where
\[
R_1=\frac{e^{\lambda}}{2\pi i}\int_{T_2}^{\infty}\Gamma(\lambda+(u-t)i)\lambda^{-(\lambda+(u-t)i)}du,
\]
\[
R_2=\frac{e^{\lambda}}{2\pi i}\int_{-\infty}^{T_1}\Gamma(\lambda+(u-t)i)\lambda^{-(\lambda+(u-t)i)}du,
\]
By Stirling's formula, it has

\begin{align}
\mathrm{Re}(\log \Gamma(\lambda+(u-t)i))=&\left(\lambda-\frac{1}{2}\right)\frac{\log(\lambda^2+(u-t)^2)}{2}-\lambda+\frac{1}{2}\log(2\pi)\notag\\
&-(u-t)\arctan\left(\frac{u-t}{\lambda}\right)+O(1/\lambda)\notag
\end{align}
And
\begin{align*}
& \mathrm{Re}(\log(\Gamma(\lambda+(u-t)i)))+\mathrm{Re}(\log(\lambda^{-(\lambda+(u-t)i)}))+\lambda \\
&=-\frac{(u-t)^2}{4\lambda^2}-\frac{(u-t)^2}{2\lambda}-\frac{1}{2}\log \lambda+\frac{1}{2}\log(2\pi)+O(1/\lambda)
\end{align*}
Hence, if $t\in [T_1+\Delta,T_2-\Delta]$, then

\begin{align*}
\left|R_1\right|&\ll \int_{T_2}^{\infty}\left|e^{\lambda}\Gamma(\lambda+(u-t)i)\lambda^{-(\lambda+(u-t)i)}\right|du \\
&\ll\lambda^{-1/2}\int_{T_2}^{\infty}\exp(-(u-t)^2/2\lambda)du\\
&\ll T^{-\alpha}.
\end{align*}
and similarly

\[
\left|R_2\right|\ll\lambda^{-1/2}\int_{-\infty}^{T_1}\exp(-(u-t)^2/2\lambda)du\ll T^{-\alpha}.
\]
if $t\le T_1-\Delta$, or $t\ge T_2+\Delta$, then
$$
\left|\omega(t,T_1,T_2)\right|\ll\lambda^{-1/2}\int_{T_1}^{T_2}\exp(-(u-t)^2/2\lambda)du\ll T^{-\alpha}.
$$

\end{proof}

\begin{mylem}
$\mathrm{(Estermann[2]})$ Let
\[
S\left(x,\frac{h}{k}\right)=\sum_{n=1}^{\infty}d(n)e\left(\frac{nh}{k}\right)e^{2\pi inx}
\]
Denote by
\[
D\left(s,\frac{h}{k}\right)=\sum_{n=1}^{\infty}d(n)e\left(\frac{nh}{k}\right)n^{-s}
\]
write $z=-2\pi ix$, then for $\mathrm{Im}\, x>0,\mathrm{Re}\,s>1,k\ge 1$, then
\begin{align}
S\left(x,\frac{h}{k}\right)&=\frac{1}{zk^{*}}(\gamma-\log z-2\log k^{*})+D\left(0,\frac{h^{*}}{k^{*}}\right)\notag\\
&-i\int\limits_{(c)}(2\pi)^{-s}\frac{\Gamma(s)k^{*2s-1}}{\sin \pi s}\left(D\left(s,\frac{\overline{h^{*}}}{k^{*}}\right)+(cos\pi s)D\left(s,-\frac{\overline{h^{*}}}{k^{*}}\right)\right)z^{s-1}ds\tag{2.4}
\end{align}
where $1<c<2$, and it has
\[
\left|D\left(0,\frac{h^{*}}{k^{*}}\right)\right|\le k^{*}(\log 2k^{*})^2 
\]

\end{mylem}

\vspace{10pt}

As usual, denote
\[
\chi(1-s)=2(2\pi)^{-s}\Gamma(s)\cos(\pi s/2).
\]
 
\begin{mylem}
Suppose that $1<c<2$, let
\[
J(y)=\frac{1}{2\pi i}\int\limits_{(c)}\Gamma(s_1-s)\lambda^{-(s_1-s)}\chi(1-s)y^{-s}ds,
\]
then
\begin{equation}
J(y)=\int_{0}^{\infty}v^{s_1}e^{-\lambda v}(e^{2\pi iyv}+e^{-2\pi iyv})\frac{dv}{v}.\tag{2.5}
\end{equation}

\end{mylem}

\begin{proof}
\[
J(y)=\frac{1}{2\pi i}\int\limits_{(c)}\Gamma(s_1-s)\lambda^{-(s_1-s)}\Gamma(s)((2\pi iy)^{-s}+(-2\pi iy)^{-s})ds
=J_1+J_2
\]
where
\begin{align}
J_1=&\frac{1}{2\pi i}\int\limits_{(c)}\Gamma(s_1-s)\lambda^{-(s_1-s)}\Gamma(s)(2\pi iy)^{-s}ds,\notag\\
J_2=&\frac{1}{2\pi i}\int\limits_{(c)}\Gamma(s_1-s)\lambda^{-(s_1-s)}\Gamma(s)(-2\pi iy)^{-s}ds.\notag
\end{align}
By the theory of Mellin Transforms(refer to [7])
\[
\frac{1}{2\pi i}\int\limits_{(c)}G(s)F(s)x^{-s}ds=\int_{0}^{\infty}f\left(\frac{1}{v}\right)g(xv)\frac{dv}{v}
\]
where
\[
f(v)=\frac{1}{2\pi i}\int\limits_{(c)}F(x)x^{-s}ds,\qquad   g(v)=\frac{1}{2\pi i}\int\limits_{(c)}G(x)x^{-s}ds.
\]
We apply this formula to $J_1$ and $J_2$ respectively.\\
For $J_1$, 
\begin{align}
f(v)=&\frac{\lambda^{-s_1}}{2\pi i}\int\limits_{(c)}\Gamma(s_1-s)(v/\lambda)^{-s}ds\notag\\
=&\frac{\lambda^{-s_1}}{2\pi i}\int\limits_{(c)}\Gamma(w)(v/\lambda)^{-(s_1-w)}ds\notag\\
=&\lambda^{-s_1}(v/\lambda)^{-s_1}e^{-\lambda/v}=v^{-s_1}e^{-\lambda/v},\notag
\end{align}
i.e.,
\[
f(1/v)=v^{s_1}e^{-\lambda v}.
\]
And
\[
g(v)=\frac{1}{2\pi i}\int\limits_{(c)}\Gamma(s)v^{-s}ds=e^{-v},
\]
i.e.
\[
g(xv)=e^{-xv}.
\]
So
\[
J_1=\int_{0}^{\infty}v^{s_1}e^{-(\lambda v+2\pi iyv)}\frac{dv}{v}
\]
Similarly,
\[
J_2=\int_{0}^{\infty}v^{s_1}e^{-(\lambda v-2\pi iyv)}\frac{dv}{v}
\]

\end{proof}

Denote by $L_{\delta}$ the straight line from $0$ to $e^{i\delta}\infty$.

\begin{mylem}
Let $s_1=\lambda + 1/2 + iu$, $0\le \delta \le \pi/2$, then
\begin{equation}
\int_{L_{\delta}}\frac{v^{s_1}e^{-\lambda v}}{v(v-1)}dv-\int_{L_{-\delta}}\frac{v^{s_1}e^{-\lambda v}}{v(v-1)}dv
=-e^{\lambda}2\pi i,\tag{2.6}
\end{equation}
And let
\[
K=\int_{L_{\delta}}\frac{v^{s_1}e^{-\lambda v}\log(-i(v-1))}{v(v-1)}dv-\int_{L_{-\delta}}\frac{v^{s_1}e^{-\lambda v}\log(i(v-1))}{v(v-1)}dv,
\]
then
\begin{equation}
K=e^{\lambda}2\pi i(\log u+c_0+c_2 u^{-2}).\tag{2.7}
\end{equation}
where $c_0$ and $c_2$ are two constants.
\end{mylem}

\begin{proof}
Equality (2.6) is followed by the residue theorem for the two integral paths form a contour with a pole at $v=1$.\\
To prove (2.7), we change the integral paths $L_{\delta}$ and $L_{-\delta}$ to the positive real axis but with an indentation around $v=1$ with
$\mathrm{Im}\,v>0$ and $\mathrm{Im}\,v<0$ respectively.\\
And let $v=e^x$, then
\begin{align}
K=&\int\limits_{\mathscr{L}_{+}}\exp(\lambda x+ixu-\lambda e^x)\frac{\log(-i(e^x-1))}{2\sinh(x/2)}dx\notag\\
&-\int\limits_{\mathscr{L}_{-}}\exp(\lambda x+ixu-\lambda e^x)\frac{\log(i(e^x-1))}{2\sinh(x/2)}dx\notag
\end{align}
where the integral path $\mathscr{L}_{+}$ is from $-\infty$ to $-\epsilon$ then along a upper semicircle $C_{\epsilon}^{+}$ to $+\epsilon$ and tend to
$+\infty$. The integral path $\mathscr{L}_{-}$ is same but with a lower semicircle $C_{\epsilon}^{-}$. Let $C_{\epsilon}$ be the union of $C_{\epsilon}^{-}$ 
and the reversal of $C_{\epsilon}^{+}$.\\
Then
\begin{align}
K=&\pi i\int_{-\infty}^{-\epsilon}\frac{\exp(\lambda x+ixu-\lambda e^x)}{2\sinh(x/2)}dx-\pi i\int_{+\epsilon}^{+\infty}\frac{\exp(\lambda x+ixu-\lambda e^x)}{2\sinh(x/2)}dx\notag\\
&-\int_{C_{\epsilon}}\exp(\lambda x+ixu-\lambda e^x)\frac{\log(i(e^x-1))}{2\sinh(x/2)}dx\notag\\
&-\frac{\pi i}{2}\int_{C_{\epsilon}^{+}\cup C_{\epsilon}^{-}}
\frac{\exp(\lambda x+ixu-\lambda e^x)}{2\sinh(x/2)}dx\notag
\end{align}
And then
\begin{align}
K=&-\int_{C_{\epsilon}}\frac{e^{-\lambda}\log x}{x}dx-R_{\epsilon}'+\pi i\int_{-\infty}^{-\epsilon}\frac{e^{-\lambda}e^{iux}}{x}dx\notag\\
&+\pi iP_{-}-\pi i\int_{+\epsilon}^{+\infty}\frac{e^{-\lambda}e^{iux}}{x}dx-\pi iP_{+}-\frac{\pi i}{2}R_{\epsilon}\notag
\end{align}
where
\begin{align}
R_{\epsilon}=&\int_{C_{\epsilon}^{+}\cup C_{\epsilon}^{-}}\frac{\exp(\lambda x+ixu-\lambda e^x)}{2\sinh(x/2)}dx\notag\\
R_{\epsilon}'=&\int_{C_{\epsilon}}\left(\exp(\lambda x+ixu-\lambda e^x)\frac{\log(i(e^x-1))}{2\sinh(x/2)}-\frac{e^{-\lambda}\log x}{x}\right)dx\notag\\
P_{-}=&\int_{-\infty}^{-\epsilon}\left(\frac{e^{iux}\exp(\lambda x-\lambda e^x)}{2\sinh(x/2)}-\frac{e^{-\lambda}e^{iux}}{x}\right)dx\notag\\
P_{+}=&\int_{+\epsilon}^{+\infty}\left(\frac{e^{iux}\exp(\lambda x-\lambda e^x)}{2\sinh(x/2)}-\frac{e^{-\lambda}e^{iux}}{x}\right)dx\notag
\end{align}
It is easy to know that $R_{\epsilon}$ and $R_{\epsilon}'\to 0$ as $\epsilon\to 0$. Let $P=\lim\limits_{\epsilon\to 0}(P_{-}-P_{+})$, then
$$
K=e^{-\lambda}2\pi i(\log u+c_0+P)
$$
where
\begin{align}
c_0=&\int_{0}^{1}\frac{1-\cos x}{x}dx-\int_{1}^{\infty}\frac{cos x}{x}dx,\notag \\
P=&\frac{1}{2}\int_{0}^{\infty}\cos(ux)\left(\frac{1}{x}-\frac{\exp(-\lambda x-\lambda(e^{-x}-1))}{2\sinh(x/2)}\right)dx\notag \\
+&\frac{1}{2}\int_{0}^{\infty}\cos(ux)\left(\frac{1}{x}-\frac{\exp(\lambda x-\lambda(e^x-1))}{2\sinh(x/2)}\right)dx\notag \\
+&\frac{i}{2}\int_{0}^{\infty}\sin(ux)\left(\frac{\exp(-\lambda x-\lambda(e^{-x}-1))}{2\sinh(x/2)}-\frac{1}{x}\right)dx\notag \\
-&\frac{i}{2}\int_{0}^{\infty}\sin(ux)\left(\frac{\exp(\lambda x-\lambda(e^x-1))}{2\sinh(x/2)}-\frac{1}{x}\right)dx\notag
\end{align}
By the partial integration, $P$ can be expressed as
\[
P=\sum_{2\le n<N}^{}c_n u^{-n}+O(u^{-N}).
\]
where $N$ is a any positive integer.

\end{proof}

The following Lemma is direct.

\begin{mylem}

\[\int_{0}^{\infty}v^{s_1-1}e^{-\lambda x}dx=\lambda^{-s_1}\int_{0}^{\infty}x^{s_1-1}e^{-x}dx=\lambda^{-s_1}\Gamma(s_1)\tag{2.8}\]

\end{mylem}

\vspace{5pt}

\begin{mylem}
Let $c=1+\eta, 0<\eta<1, s_1=\lambda+1/2+iu,T\le u\le 2T, \delta=1/T $, and let
$$
W=\int\limits_{L_{\delta}}\int\limits_{(c)}e^{\lambda}v^{s_1}e^{-\lambda v}(1+|s|)\frac{\Gamma(s)}{\sin \pi s}(\cos \pi s)e^{-\pi is/2}(v-1)^{s-1}ds
\frac{dv}{v}.
$$
Then
\[
W\ll_{\epsilon,\eta}T^{\epsilon}(\log T/T)^{c}.\tag{2.9}
\]
and the estimation is also holds in the cases the term $cos\pi s$ is removed, or $L_{\delta}$ is replaced by $L_{-\delta}$ and $e^{-\pi is/2}$ by
$e^{\pi is/2}$.

\end{mylem}

\begin{proof}
Let $v=xe^{i\delta}$, then
\[|\Gamma(s)|\ll (1+|t|)^{c-1/2}e^{-\pi |t|/2},\]
\[\left|\frac{\cos \pi s}{\sin \pi s}\right|\ll 1,\]
\[|e^{-\pi s/2}|=e^{\pi t/2},\]
\[|v^{s_1}|\ll x^{\lambda+1/2},\]
\[ 1/|v|\ll 1/x,\]
\[|e^{-\lambda v}|=e^{-\lambda'x},\]
where $\lambda'=\lambda \cos\delta$, and clearly, $\lambda-T^{-\epsilon}/2<\lambda'<\lambda$.\\
And
\[|(v-1)^{s-1}|=|v-1|^{c-1}e^{-t\arg(v-1)}=a(x,\delta)\exp(-tb(x,\delta))\]
where
\[a(x,\delta)=((x-a)^2+2x(1-\cos\delta))^{(c-1)/2}\]
\[b(x,\delta)=\arctan\frac{x\sin\delta}{x\cos\delta-1}\]
Suppose that $\beta\ge 10$, let $\theta=(2\beta\log T/\lambda)^{1/2}$, there are the following estimates
\[
a(x,\delta)\ll
\begin{cases}
  \,1, & \mbox{if }x\le 1-\theta,\\
  \,\theta^{c-1}, & \mbox{if }|x-1|\le \theta,\\
  \,x^{c-1}, & \mbox{if } x\ge 1+\theta.
\end{cases}
\]
and
\[
b(x,\delta)
\begin{cases}
 \, >\delta, & x>0, \\
 \, \gg \theta^{-1}\delta, & |x-1|\le\theta,\\
 \, \le \pi-C\theta^{-1}\delta, & |x-1|\le\theta, \\
 \, \le \pi-Cx\delta, & 0<x\le 1-\theta,\\
 \, \le \delta\theta^{-1}, & x\ge 1+\theta.
\end{cases}
\]
where $C$ is a constant. Hence
\[
\int_{-\infty}^{\infty}(1+|t|)^{c+1/2}e^{\pi(t-|t|)/2}e^{-tb(x,\delta)}dt\ll
\begin{cases}
  \delta^{-c-3/2}, & x\ge 1+\theta, \\
  (\theta\delta^{-1})^{c+3/2}, & |x-1|\le\theta, \\
  (\delta x)^{-c-3/2}, & 0<x\le 1-\theta.
\end{cases}
\]
Divide the $x$ integral into three pieces $W_1,W_2$ and $W_3$ with $ 0 \le x \le 1-\theta $, $1-\theta\le x\le 1+\theta$ and $1+\theta\le x\le \infty$ 
respectively, then there are
\begin{align*}
  &W = W_1+W_2+W_3. \\
  &W_1 \ll \delta^{-5/2-\eta} e^{\lambda}\int_{0}^{1-\theta}x^{\lambda-c-2}e^{-{\lambda}' x}dx\ll \delta^{-5/2-\eta}T^{-\beta}\\
  &W_2 \ll T^{\epsilon}\theta^{c} \\
  &W_3 \ll \delta^{-5/2-\eta} e^{\lambda}\int_{1+\theta}^{\infty}x^{\lambda+c-3/2}e^{-{\lambda}' x}dx\ll \delta^{-5/2-\eta}T^{-\beta}
\end{align*}
  
\end{proof}

\begin{mylem}
Let $s_1=\lambda+1/2+iu$, define
$$
\mathfrak{g}(u)=\frac{e^{\lambda}}{2\pi i}\int\limits_{(1/2)}\Gamma(s_1-s)\lambda^{-s_1+s}\zeta(s)\chi(1-s)\bar{A}(s)A(1-s)ds.
$$
Then for $T\le u\le 2T$, there is
\[
\mathfrak{g}(u)=\sum\limits_{1\le h,k\le M}\frac{a(h)\overline{a(k)}}{hk}(h,k)\left(\log\frac{u(h,k)^2}{2\pi hk}+b_0+O(u^{-2})\right)+\mathscr{E}_{0}\tag{2.10}
\]
where $\mathscr{E}_{0}\le VT^{-c+\epsilon}$, and $b_0$ is a constant.

\end{mylem}

\begin{proof}
We move the integral path from $(1/2)$ to $(c),c=1+\eta,0<\eta<1 $, the residue at $s=1$ is
\[
R=e^{\lambda}\Gamma(\lambda-1/2-iu)\lambda^{-(\lambda-1/2-iu)}\zeta(0)A(1)\bar{A}(0)\ll M^{1+\epsilon}T^{-\alpha}\ll T^{-\alpha+1}
\]
Hence,
\[
\mathfrak{g}(u)=\hat{\mathfrak{g}}(u)-R
\]
where $\hat{\mathfrak{g}}(u)$ is the one of $\mathfrak{g}(u)$ moved in the new integral path.
\begin{align*}
 \hat{\mathfrak{g}}(u)  & =\sum_{1\le h,k \le M}\frac{a(h)\overline{a(k)}}{k}e^{\lambda}\sum\limits_{n=1}^{\infty}d(n)J(\frac{nh}{k}) \\
   & =\sum_{1\le h,k \le M}\frac{a(h)\overline{a(k)}}{k}\sum\limits_{n=1}^{\infty}d(n)e^{\lambda}\int_{0}^{\infty}v^{s_1}e^{-\lambda v}
       (e^{2\pi inh/k}+e^{-2\pi inh/k})\frac{dv}{v}\\
   &=\sum_{1\le h,k \le M}\frac{a(h)\overline{a(k)}}{k}(I_1+I_2)
\end{align*}

where
\[
I_1=M_1+R_1-iE_1,\quad  I_2=M_2+R_2-iE_2.
\]
\begin{align*}
   & M_1=e^{\lambda}\int\limits_{L_{\delta}}v^{s_1}e^{-\lambda v}\frac{\gamma-\log(-2\pi i(v-1)h/k)-2\log k^{*}}{-2\pi ih^{*}}\frac{dv}{v} \\
   & R_1=D\left(0,\frac{h^{*}}{k^{*}}\right)e^{\lambda}\int\limits_{L_{\delta}}v^{s_1}e^{-\lambda v}\frac{dv}{v}  \\
   & E_1=e^{\lambda}\int\limits_{L_{\delta}}v^{s_1}e^{-\lambda v}F_1(v)\frac{dv}{v} \\
   & F_1(v)=\int\limits_{(c)}(2\pi)^{-s}\frac{\Gamma(s)}{\sin \pi s}\left(D\left(s,\frac{h^{*}}{k^{*}}\right)+(\cos\pi s)D\left(s,-\frac{\overline{h^{*}}}{k^{*}}\right)\right)\\
   & \qquad   \qquad    \qquad  \qquad  \qquad  \qquad    \qquad   \quad      \cdot(-2\pi ih^{*}k^{*}(v-1))^{s-1}ds
\end{align*}
And
\begin{align*}
   & M_2=e^{\lambda}\int\limits_{L_{-\delta}}v^{s_1}e^{-\lambda v}\frac{\gamma-\log(-2\pi i(v-1)h/k)-2\log k^{*}}{2\pi ih^{*}}\frac{dv}{v} \\
   & R_2=D\left(0,-\frac{h^{*}}{k^{*}}\right)e^{\lambda}\int\limits_{L_{-\delta}}v^{s_1}e^{-\lambda v}\frac{dv}{v}  \\
   & E_2=e^{\lambda}\int\limits_{L_{-\delta}}v^{s_1}e^{-\lambda v}F_2(v)\frac{dv}{v} \\
   & F_2(v)=\int\limits_{(c)}(2\pi)^{-s}\frac{\Gamma(s)}{\sin \pi s}\left(D\left(s,-\frac{h^{*}}{k^{*}}\right)+(\cos\pi s)D\left(s,\frac{\overline{h^{*}}}{k^{*}}\right)\right)\\
   & \qquad   \qquad    \qquad  \qquad  \qquad  \qquad    \qquad   \quad      \cdot(2\pi ih^{*}k^{*}(v-1))^{s-1}ds
\end{align*}
By Lemma 2.4,
$$
M_1+M_2=\frac{1}{h^{*}}\left(\log\frac{u(h,k)^2}{2\pi hk}+\gamma+c_0+O(u^{-2})\right)
$$
and by Lemma 2.5
$$
R_1+R_2\ll \mathrm{Re}D\left(0,\frac{h^{*}}{k^{*}}\right)T^{-\alpha}\ll M(\log M)^2T^{-\alpha}\ll T^{-\alpha+1}
$$
So
$$
\sum\limits_{1\le h,k \le M}\frac{a(h)\overline{a(k)}}{k}(R_1+R_2)\ll M^{1+2\epsilon}\log M T^{-\alpha+1}\ll T^{-\alpha+2} 
$$
The other error term is
$$
\sum\limits_{1\le h,k \le M}\frac{a(h)\overline{a(k)}}{k}(E_1+E_2),
$$
which may be written as a sum of four terms of a typical one is that
$$
Z=\int\limits_{L_{\delta}}\int\limits_{(c)}G(v,s_1,s)\mathscr{M}(s)dsdv
$$ 
where
$$
G(v,s_1,s)=e^{\lambda}v^{s_1}e^{-\lambda v}\frac{\Gamma(s)}{\sin \pi s}(2\pi)^{-2s}(-2\pi i(v-1))^{s-1}v^{-1},
$$

$$
\mathscr{M}(s)=\sum_{n=1}^{\infty}\frac{d(n)}{n^{s}}\sum_{h,k\le M}^{}\frac{a(h)\overline{a(k)}}{(hk)^{1-s}}(h,k)^{(1-2s)}e\left(\frac{n\overline {h^{*}}}{k^{*}} \right).
$$
By Lemma 2.6,
$$
Z\ll VT^{\epsilon}\theta^{c}.
$$

\end{proof}

\noindent The Proof of Theorem 1.1.

\begin{proof}
By Lemma 2.1, it has
\begin{align*}
& I(T,2T)= \int_{T}^{2T}\omega(t,T-\Delta,2T+\Delta)|\zeta A(1/2+it)|^2dt+o(1) \\
& = \frac{e^{\lambda}}{2\pi}\int_{T}^{2T}\int_{T-\Delta}^{2T+\Delta}\Gamma(\lambda+(u-t)i)\lambda^{-(\lambda+(u-t)i)}du|\zeta A(1/2+it)|^2dt+o(1) \\
& = \frac{e^{\lambda}}{2\pi}\int_{T-\Delta}^{2T+\Delta}\int_{T}^{2T}\Gamma(\lambda+(u-t)i)\lambda^{-(\lambda+(u-t)i)}|\zeta A(1/2+it)|^2dtdu+o(1) \\
& \le \int_{T-\Delta}^{2T+\Delta}\mathfrak{g}(u)du+o(1) \\
& = \int_{T-\Delta}^{2T+\Delta}\sum\limits_{1\le h,k\le M}\frac{a(h)\overline{a(k)}}{hk}(h,k)\left(\log\frac{u(h,k)^2}{2\pi hk}+b_0+O(u^{-2})\right) +O(\mathscr{E}_{0}T) 
\end{align*}
\begin{align*}
& = T\sum\limits_{1\le h,k\le M}\frac{a(h)\overline{a(k)}}{hk}(h,k)\left(\log\frac{T(h,k)^2}{2\pi hk}+b_0-1+2\log 2\right) +O(\mathscr{E}_{0}T)  \\
&\qquad \qquad \qquad \qquad \qquad \qquad \qquad + O\bigg(\Delta\log T\sum\limits_{1\le h,k\le M}\frac{a(h)\overline{a(k)}}{hk}(h,k)\bigg)
\end{align*}
and
$$
\sum_{1\le h,k\le M}\frac{(h,k)}{hk}\le \sum_{1\le t\le M}t\, \big(\sum_{\substack{h\le M\\ t|h }}h^{-1}\big)^2\ll \log^{3}M
$$
that is, the last term is $O(\Delta\log^4 M M^{2\epsilon})$. \\
Then replacing $T$ by $T/2^k,1\le k\le \log T$, and summing, it follows

$$
I\le T\sum_{h,k\le M}^{}\frac{a(h)}{h}\frac{\overline{a(k)}}{k}(h,k)\left(\log\frac{T(h,k)^2}{2\pi hk}+b_0 -1\right)+\tilde{\mathscr{E}}
$$ 

On the other hand, similarly it has

\begin{align*}
& I(T,2T)= \int_{T}^{2T}\omega(t,T+\Delta,2T-\Delta)|\zeta A(1/2+it)|^2dt+o(1) \\
& = \frac{e^{\lambda}}{2\pi}\int_{T}^{2T}\int_{T+\Delta}^{2T-\Delta}\Gamma(\lambda+(u-t)i)\lambda^{-(\lambda+(u-t)i)}du|\zeta A(1/2+it)|^2dt+o(1) \\
& = \frac{e^{\lambda}}{2\pi}\int_{T+\Delta}^{2T-\Delta}\int_{T}^{2T}\Gamma(\lambda+(u-t)i)\lambda^{-(\lambda+(u-t)i)}|\zeta A(1/2+it)|^2dtdu+o(1) \\
& \ge \int_{T+\Delta}^{2T-\Delta}\mathfrak{g}(u)du+o(1) 
\end{align*}
Hence, it will lead same bound as the upper bound above. As for the fact $b_0=2\gamma$, it may be followed from the known result of Ingham [5].

\end{proof}

It should be mentioned that the integrant of $\omega(t,T_1,T_2)$ in general is not as the one of $w(t,T_1,T_2)$ a positive real number, nevertheless,
its argument is about $\frac{t-u}{2\lambda}-\frac{(t-u)^3}{6\lambda^2}$, which is very small in the context, so can be approximately viewed as a positive real number, and the deduction
above is valid as the one in [1].

\vspace{1cm}

Acknowledgement:  The article was communicated with Prof. Heath-Brown and Prof. Conrey, and they gave me some good advisements.

\end{document}